\documentclass[a4paper]{article}
\usepackage{amsthm,amsfonts,amsmath,amssymb}
\usepackage[cp1251]{inputenc}
\usepackage[english]{babel}
\usepackage[final]{graphicx}
\usepackage{setspace}
\usepackage[12pt]{extsizes}
\begin{document}
\newcommand{\per}{{\rm per}}
\newcommand{\Per}{{\rm Per}}
\newtheorem{teorema}{Theorem}
\newtheorem{lemma}{Lemma}
\newtheorem{utv}{Proposition}
\newtheorem{svoistvo}{Property}
\newtheorem{sled}{Corollary}
\newtheorem{con}{Conjecture}
\newtheorem{zam}{Remark}

\author{A. A. Taranenko}
\title{Transversals, plexes, and multiplexes in iterated quasigroups \thanks{The author was supported by the Moebius Contest Foundation for Young Scientists.}
}
\date{}

\maketitle

\begin{abstract}
A $d$-ary quasigroup of order $n$ is a $d$-ary operation over a set of cardinality $n$ such that the Cayley table of the operation is a $d$-dimensional latin hypercube of the same order. Given a binary quasigroup $G$, the $d$-iterated quasigroup $G^{\left[d\right]}$ is a $d$-ary quasigroup that is a $d$-time composition of $G$ with itself. A $k$-multiplex (a $k$-plex) $K$ in a $d$-dimensional latin hypercube $Q$ of order $n$ or in the corresponding $d$-ary quasigroup is a multiset (a set) of $kn$ entries such that each hyperplane and each symbol of $Q$ is covered by exactly $k$ elements of $K$. A transversal is a 1-plex. 

In this paper we prove that there exists a constant $c(G,k)$ such that if a $d$-iterated quasigroup $G$ of order $n$ has a $k$-multiplex then for large $d$ the number of its $k$-multiplexes is asymptotically equal to $c(G,k) \left(\frac{(kn)!}{k!^n}\right)^{d-1}$. As a corollary we obtain that if the number of transversals in the Cayley table of a $d$-iterated quasigroup $G$ of order $n$ is nonzero then asymptotically it is $c(G,1)  n!^{d-1}$.  

In addition, we  provide limit constants and recurrence formulas for the numbers of transversals in two iterated quasigroups of order $5$, characterize a typical $k$-multiplex and estimate numbers of partial $k$-multiplexes and transversals in $d$-iterated quasigroups.
\end{abstract}

\section{Definitions and preliminaries}

Let $n,d \in \mathbb N$,  $I_n^d= \left\{ (\alpha_1, \ldots , \alpha_d):\alpha_i \in \left\{1,\ldots,n \right\}\right\}$ , and let $I_n^1 = \left\{1,\ldots,n \right\}$. Denote by $I_{\left\{n,k\right\}}$ the multiset of size $kn$ over the set $\left\{1, \ldots, n\right\}$ in which each of $n$ symbols appears exactly $k$ times.

A \textit{$d$-dimensional matrix $A$ of order $n$} is an array $(a_\alpha)_{\alpha \in I^d_n}$, $a_\alpha \in\mathbb R$. The \textit{support} of a matrix $A$ is the set of indices with nonzero values.

Let $k\in \left\{0,\ldots,d\right\}$. A \textit{$k$-dimensional plane} in $A$ is the submatrix of $A$ obtained by fixing $d-k$ indices and letting the other $k$ indices vary from 1 to $n$. A 1-dimensional plane is said to be a \textit{line}, and a $(d-1)$-dimensional plane is a \textit{hyperplane}.

A \textit{$d$-dimensional latin hypercube $Q$ of order $n$} is a $d$-dimensional matrix of order $n$ filled by $n$ symbols so that all symbols within each line are distinct. 2-Dimensional latin hypercubes are known as \textit{latin squares}.

A \textit{$d$-dimensional MDS code $M$ of order $n$} is a $d$-dimensional (0,1)-matrix of order $n$ such that each line contains exactly one unity entry.

A \textit{$d$-ary quasigroup $f$ of order $n$} is a function $f: I_n^d \rightarrow I_n^1$ such that the equation $x_0 = f(x_1, \ldots, x_n)$ has a unique solution for any one variable if all the other $n$ variables are specified arbitrarily.  A \textit{binary} quasigroup of order $n$ is a binary operation $*$ over a set $I_n^1$ with the following property: for each $a_0, a_1, a_2 \in I_n^1$ there exist unique $x_1, x_2 \in I_n^1$ such that both $a_0 = a_1 * x_2$ and $a_0 = x_1 * a_2$ hold. 

There are natural correspondences between MDS codes, latin hypercubes, and quasigroups of the same order. The Cayley table of a $d$-ary quasigroup $f$ of order $n$ is a $d$-dimensional latin hypercube $Q(f)$ of order $n$, and vice versa, every $d$-dimensional latin hypercube can be considered as the Cayley table of some $d$-ary quasigroup.  The graph $\left\{(x_0, x_1, \ldots, x_d) |~ x_0 =f(x_1, \ldots, x_d) \right\}$ of a quasigroup $f$ is the set of unity entries of the $(d+1)$-dimensional MDS code $M(f)$ of order $n$. The correspondence between a $d$-dimensional latin hypercube $Q$ and the $(d+1)$-dimensional MDS code $M(Q)$ is given by the following rule: an entry $q_{\alpha_1, \ldots, \alpha_{d}}$ of a latin hypercube $Q$ equals $\alpha_{d+1}$ if and only if an entry $m_{\alpha_1, \ldots, \alpha_{d+1}}$ of the MDS code $M(Q)$ equals 1.
The Cayley table of every binary quasigroup $G$ is a latin square $Q(G)$ that corresponds to some 3-dimensional MDS code $M(G)$. 

Although all these three approaches (via quasigroups, latin hypercubes, and MDS codes) are equivalent, in different environments it is more  customary to use only one of them. During this paper we will change between these three concepts repeatedly because  algebraic properties of quasigroups are needed for the proofs, the whole research is motivated by problems for latin squares and hypercubes, but approach via MDS codes is often more symmetric and convenient.

Given  a binary quasigroup $G$ of order $n$ defined by a binary operation $*$,  the \textit{$d$-iterated quasigroup $G$} denoted by  $G^{\left[d\right]}$ is the $d$-ary quasigroup of order $n$ such that
$$x_0 = G^{\left[d\right]} (x_1, \ldots, x_d) \Leftrightarrow (\ldots((x_0 * x_1) * x_2)* \ldots * x_{d-1})*x_d = 1.$$
The latin hypercube that is the Cayley table of the $d$-iterated quasigroup $G$ we denote by $Q(G^{\left[d\right]})$ and the corresponding $(d+1)$-dimensional MDS code denote by $M(G^{\left[d\right]})$.

Let $A$ be a $d$-dimensional matrix of order $n$. A  multiset $K$ of $kn$ indices $\left\{\alpha^1, \ldots, \alpha^{kn}\right\}$  is called a \textit{$k$-multiplex}  if each hyperplane of $A$ contains exactly $k$ elements of $K$. A $k$-multiplex $K$ is called a \textit{$k$-plex} if all elements of $K$ are different (namely, $K$ is a set). 

A 1-plex (or a 1-multiplex) is a set of $n$ indices such that there is exactly one index in each hyperplane. Such sets of indices are known as \textit{diagonals} of multidimensional matrices. Note that union of every $k$ diagonals is a $k$-multiplex, union of $k$ mutually disjoint diagonals is a $k$-plex, but not every $k$-multiplex or $k$-plex can be partitioned  into diagonals.

We will say that a $k$-multiplex $K$ is \textit{indivisible} if there are no $k_1$-multiplexes $K_1$ and $k_2$-multiplexes $K_2$ such that $K$ is the union of $K_1$ and $K_2$; otherwise $K$ is called \textit{divisible}.  A $k$-multiplex $K$ in a $d$-dimensional matrix of order $n$ is \textit{disconnected} if $K$ is the union of $k$-multiplexes $K_1$ and $K_2$ in matrices of smaller orders $n_1$ and $n_2$ ($n = n_1+n_2$);  otherwise $K$ is said to be \textit{connected}.

For a $d$-dimensional matrix $A$ of order $n$, define the \textit{$k$-permanent} of a matrix $A$ to be
$$\per_k A = \sum\limits_{k-plexes~K} \prod\limits_{\alpha \in K} a_{\alpha},$$
and the \textit{$k$-multipermanent} of $A$ to be
$$\Per_k A = \sum\limits_{k-multiplexes~K} \prod\limits_{\alpha \in K} a_{\alpha}.$$
Using term ``permanent'' for these objects is explained by the fact that the 1-permanent (or the 1-multipermanet) is exactly the permanent of multidimensional matrices that was studied in detail in~\cite{myobz}.

It is easy to see that if $A$ is a $(0,1)$-matrix then the $k$-permanent of $A$ is exactly the number of $k$-plexes in the support of $A$ and the $k$-multipermanent of $A$ is the number of $k$-multiplexes over the support of $A$. Also, for every $d$-dimensional (0,1)-matrix $A$ of order $n$ we have $$\per_k A \leq \Per_k A \leq \left( \frac{(kn)!}{k!^n} \right)^{d},$$ since every $k$-plex is a $k$-multiplex and since in each of $d$ positions of indices from a multiplex  we have a permutation of the multiset $I_{\left\{n,k\right\}}$.

A diagonal in a $d$-dimensional latin hypercube $Q$ whose entries contain all different symbols is said to be a \textit{transversal}. A \textit{$k$-plex} in a $d$-dimensional latin hypercube $Q$ of order $n$ is a selection of $kn$ different indices of $Q$ in which each hyperplane and symbol is represented precisely $k$ times. Analogically, we can define $k$-multiplexes in latin hypercubes. From definitions it follows that every transversal in a latin hypercube $Q$ is a diagonal in the MDS code $M(Q)$, every $k$-plex ($k$-multiplex) in $Q$ corresponds to a $k$-plex (a $k$-multiplex) in $M(Q)$, and reverse. For $d$-ary quasigroups we define transversals, $k$-plexes, and $k$-multiplexes so that they coincide with those in latin hypercubes.

For a $k$-multiplex $ K = \left\{\alpha^1, \ldots, \alpha^{kn}\right\}$ we denote by $K_j$ the $kn$-vector $(\alpha_j^1, \ldots, \alpha_j^{kn})$ that is the $j$th component of $K$-multiplex $K$. Given a binary quasigroup $*$ and two $m$-vectors $U$ and $V$, a component-wise product of $U$ and $V$ is
$$U * V = W \Leftrightarrow u_i * v_i = w_i \mbox{ for all } i = 1, \ldots, m.$$
We define $E$ to be a vector with unity components.

\section{Motivation and the main result}

Transversals and plexes in latin squares have been widely studying for the last decades but many important questions on existence are not solved yet.
One of the most known conjectures on transversals belongs to Ryser~\cite{ryser}.

\begin{con}[Ryser] \label{hypryser}
Every latin square of odd order has a transversal.
\end{con}

According to~\cite{handbook}, the following conjecture for 2-plexes in latin squares was proposed by Rodney.

\begin{con}[Rodney] \label{hyprod}
Every latin square has a $2$-plex.
\end{con}

For the Cayley tables of groups this conjecture was proved by Vaughan-Lee and Wanless.

\begin{teorema}[\cite{2plexgroups}] \label{group2plex}
If $G$ is a group then the latin square $Q(G)$ has a $2$-plex.
\end{teorema}

For larger $k$, $k$-plexes in latin squares are weakly investigated. 
One of the possible ways to generalize plexes was considered by Pula in~\cite{pula}, and a comprehensive survey of other results on transversals and $k$-plexes in latin squares is given in~\cite{wanless}. 

Investigation of transversals in latin hypercubes started in only last years. In~\cite{wanless} Wanless generalized the Ryser's conjecture for latin hypercubes.
\begin{con}[Wanless]
Every latin hypercube of odd dimension or odd order has a transversal.
\end{con}
It is known that if $n$ and $d$ are both even then the Cayley table of the $d$-iterated group $\mathbb{Z}_n$ has no transversals~\cite{myobz, wanless}. Moreover, using the same technique it is easy to prove the analogous statement for $k$-multiplexes.

\begin{utv} \label{nooddplexesZ}
Let $n$ and $d$ be even and $k$ be odd. Then the $d$-dimensional latin hypercube $Q(\mathbb{Z}^{\left[d\right]}_n)$ has no $k$-multiplexes. 
\end{utv}

\begin{proof}
Assume that a multiset of indices $K = \left\{\alpha^1, \ldots, \alpha^{kn} \right\}$ is a $k$-multiplex in the latin hypercube $Q(\mathbb{Z}^{\left[d\right]}_n)$. Consider the sum 
$$S = \sum\limits_{i=1}^{kn} \sum\limits_{j=0}^{d} \alpha_j^i.$$

Since for each $\alpha^i \in K$ it holds $\sum\limits_{j=0}^d \alpha^i_j \equiv 0 \mod n$, we have $S \equiv 0 \mod n.$ On the other hand, for each $j \in \left\{0, \ldots, d\right\}$ we have $\sum\limits_{i=1}^{kn} \alpha^i_j = k \frac{n(n+1)}{2}.$ Therefore,
$$S = \sum\limits_{j=0}^d \sum\limits_{i=1}^{kn}  \alpha_j^i = \sum\limits_{j=0}^d k \frac{n(n+1)}{2} = k(d+1) \frac{n(n+1)}{2} \not\equiv 0 \mod n,$$
because $n$ and $d$ are even and $k$ is odd. 
\end{proof}

The numbers of transversals in all latin hypercubes of orders 2 and 3 are found in~\cite{myobz}, and the numbers of transversals in $d$-iterated groups of order 4 are calculated in~\cite{myquasi}. Also, in~\cite{myquasi} it is proved that for all odd $d$ a $d$-iterated quasigroup has transversals, and a lower bound on their number is obtained.     

An asymptotic  behavior of the maximum number of transversals in latin hypercubes of fixed dimension and large order was found in~\cite{glebov,mytrans}. In~\cite{eberhyper} it was proved that if for an abelian group $G$ of order $n$ the latin hypercube $Q(G^{\left[d\right]})$ has a transversal then for large order $n$  and fixed dimension $d$ the latin hypercube  $Q(G^{\left[d\right]})$ has asymptotically maximal number of transversals.

One of the main results of this paper is that the analogous statement holds for the Cayley tables of $d$-iterated quasigroups of large dimension. Namely, it will be a special case of the following theorem.

\begin{teorema} \label{quasilowmulti}
Let $G$ be a binary quasigroup of order $n$ and let $M(G^{\left[d-1\right]})$ be the $d$-dimensional MDS code of the  $(d-1)$-iterated quasigroup $G$.
\begin{enumerate}
\item For all even $d$ the MDS code $M(G^{\left[d-1\right]})$ has a nonzero $k$-multipermanent. If for some odd $d'$ we have that $\Per_k M(G^{\left[d'-1\right]})$ is positive then $\Per_k M(G^{\left[d-1\right]}) $ is greater than zero for all $d \geq d'$.
\item  There exists a constant $c(G,k) >0$ such that  
$$\lim\limits_{d \rightarrow \infty} \frac{\Per_k  M(G^{\left[d-1\right]})}{\left( \frac{(kn)!}{k!^n} \right)^{d-2}} = c(G,k) ,$$
where limit is taken over all $d$ for which $\Per_k  M(G^{\left[d-1\right]})$ is nonzero.
\end{enumerate}
\end{teorema}

\begin{sled} \label{translow}
Let $G$ be a binary quasigroup of order $n$ and let $Q(G^{\left[d\right]})$ be the $d$-dimensional latin hypercube that is the Cayley table of the $d$-iterated quasigroup $G$.
\begin{enumerate}
\item For all odd $d$ the latin hypercube $Q(G^{\left[d\right]})$ has a transversal. If for some even $d'$ we have that $Q(G^{\left[d'\right]})$ has transversals then the latin hypercubes $Q(G^{\left[d\right]}) $ have transversals for all $d \geq d'$.
\item  There exists a constant $c(G,1) >0$ such that if the latin hypercube $Q(G^{\left[d\right]})$ has transversals then for large $d$ the number of transversals is asymptotically equal to $c(G,1)n!^{d-1}$. 
\end{enumerate}
\end{sled}

\section{Corollaries of the main result}

Before the proof of Theorem~\ref{quasilowmulti} we deduce several other corollaries. 
For this purpose we need the following auxiliary lemmas about plexes, multiplexes, and $k$-multipermanents.

\begin{lemma} \label{plexsvoi}
\begin{enumerate}
\item If $A$ is a $d$-dimensional $(0,1)$-matrix $A$ of order $n$ and if for some $k$ we have $\Per_k A > 0$ then for all integer $m\geq 1$ it holds $\Per_{km} A > 0$.
\item Let $G$ be a binary quasigroup of order $n$. If for some $k$ and $d$ we have $\Per_k M(G^{\left[d\right]}) > 0$  then for all integer $m\geq 0$ it holds $\Per_k M(G^{\left[d+2m\right]}) > 0$. Also, it holds $ \per_k M(G^{\left[d+2m\right]}) \geq \left(\frac{(kn)!}{k!^n}\right)^m \per_k M(G^{\left[d\right]})$. 
\end{enumerate}
\end{lemma}

\begin{proof}
1. It follows from the fact that a union of $m$ $k$-multiplexes is a $km$-multiplex.

2. Let $K = \left\{\alpha^1, \ldots, \alpha^{kn}\right\}$ be a $k$-multiplex (a $k$-plex) over the support of $M(G^{\left[d\right]})$. By definition of $k$-multiplex,  we have
$$(\ldots((K_0 * K_1) * K_2)* \ldots * K_{d-1})*K_d = E.$$
Let $U = (u^1, \ldots, u^{kn})$ be a $kn$-vector that is a permutation of the multiset $I_{\left\{ n, k\right\}}$. Define vector $V$ from the relation $((E * U )* V) = E.$

Since $*$ is the operation of the quasigroup $G$, the vector $V$ is a permutation of the multiset $I_{\left\{ n, k\right\}}$. Consequently, for each $i$ the index $ \beta^i = (\alpha_1^i, \ldots, \alpha_d^i, u^i, v^i)$ belongs to the support of the $(d+2)$-iterated group $G^{\left[d+2\right]}$. Therefore, $K' = \left\{\beta^1, \ldots, \beta^{kn}\right\}$ is a $k$-multiplex in the MDS code $M(G^{\left[d+2\right]})$. It only remains to note that if $K$ is a $k$-plex in $M(G^{\left[d\right]})$ then all constructed $k$-plexes  $K'$ in the MDS code $M(G^{\left[d+2\right]})$ are different.
\end{proof}

We will say that a $k$-multiplex $K$ is a \textit{true} $k$-multiplex if it is not a $k$-plex. 

\begin{lemma} \label{uppers}
Let $A$ be a $d$-dimensional matrix of order $n$. Then the following holds.
\begin{enumerate}
\item The number of $k$-multiplexes in the matrix $A$ is not greater than $\left( \frac{(kn)!}{k!^n} \right)^{d}$.
\item The number of true $k$-multiplexes is not greater than $\left( n \frac{(kn-2)!}{k!^{n-1} (k-2)!} \right)^{d}$. 
\item The number of disconnected $k$-multiplexes that are a union of two $k$-multiplexes of orders $n_1$ and $n_2$ $(n_1 + n_2 = n)$ is not greater than $\left({n \choose n_1} \frac{(kn_1)!(kn_2)!}{k!^n}\right)^d.$ 
\item Let $S$ be the minimal size of supports of hyperplanes of the matrix $A$. If $k > S$ then there are no $k$-plexes on the support of $A$. 
\end{enumerate}
\end{lemma}

\begin{proof}
1. As it was mentioned before, in each of $d$ positions of indices of a $k$-multiplex we have a permutation of the multiset $I_{\left\{n,k\right\}}$.

2.  If $K = \left\{\alpha^1, \ldots, \alpha^{kn}\right\}$ is a $k$-multiplex but is not a $k$-plex then $K$ contains at least two identical elements. Without loss of generality,  suppose that indices $\alpha^1$ and $\alpha^2$ are the same. For constructing a true $k$-multiplex $K$, in each of $d$ positions we have $n$ possibilities to choose a symbol for indices $\alpha^1$ and $\alpha^2$ and for other $kn-2$ indices we take a permutation of the multiset of $kn-2$ elements in which each of $n-1$ symbols appears exactly $k$ times and the last symbol appears $k-2$ times.  

3.  If a $k$-multiplex $K = \left\{\alpha^1, \ldots, \alpha^{kn}\right\}$ is a union of two $k$-multiplexes of orders $n_1$ and $n_2$ then $K$ can be partitioned into two  submultisets $K_1$ and $K_2$ of sizes $k_1n$ and $k_2n$ such that any two indices $\alpha^i \in K_1$ and $\alpha^j \in K_2$ differ at all positions. For constructing such disconnected $k$-multiplex $K$, in each of $d$ positions we choose which $n_1$ of $n$ symbols will be used for elements of $K_1$ and then we take a permutation of the multiset $I_{\left\{n_1, k\right\}}$ over chosen symbols and a permutation of the multiset $I_{\left\{n_2, k\right\}}$ over other symbols.

4. By definition, each hyperplane of $A$ contains exactly $k$ different elements of a $k$-plex, so if the support of $A$ in one of hyperplanes is less than $k$ then $A$ has no $k$-plexes.  
\end{proof}

Using this lemma and Theorem~\ref{quasilowmulti} we prove that for large $d$ a typical $k$-multiplex in a $d$-iterated quasigroup is a connected indivisible $k$-plex. 

\begin{teorema} \label{typicalmulti}
Let $G$ be a binary quasigroup of order $n$ and let $M(G^{\left[d\right]})$ be the $(d+1)$-dimensional MDS code of the $d$-iterated quasigroup $G$.
Then the ratio of the number of connected indivisible $k$-plexes to the number of $k$-multiplexes in  $M(G^{\left[d\right]})$ tends to one as $d$ tends to infinity. In other words, there is a constant $c(G,k) > 0$ such that the  ratio of the number of connected indivisible $k$-plexes to  $\left( \frac{(kn)!}{k!^n} \right)^{d-1}$ tends to $c(G,k)$ as $d \rightarrow \infty$  and where limit is taken over all $d$ for which $\Per_k M(G^{\left[d\right]}) >0$.
\end{teorema}

\begin{proof}
We prove that the ratios of the numbers of true $k$-multiplexes, divisible $k$-multiplexes, and disconnected $k$-multiplexes to the number of all $k$-multiplexes in $M(G^{\left[d\right]})$ (if it is nonzero) tends to zero as $d$ tends to infinity that implies the statement of the theorem. Recall that, by Theorem~\ref{quasilowmulti}, there exists a constant $c > 0$ such that if $\Per_k M(G^{\left[d\right]}) >0$ then 
$$\Per_k M(G^{\left[d\right]}) \geq c \left( \frac{(kn)!}{k!^n} \right)^{d}.$$

1. By Lemma~\ref{uppers}, the number of true $k$-multiplexes in $M(G^{\left[d\right]})$ is not greater than $\left( n \frac{(kn-2)!}{k!^{n-1} (k-2)!} \right)^{d}$. Therefore, the ratio of this number to the number of all $k$-multiplexes in $M(G^{\left[d\right]})$ is not greater than $\frac{1}{c} \left(\frac{1}{k^2(k-1)(kn-1)}\right)^d$ that tends to zero as $d \rightarrow \infty$.

2. By Lemma~\ref{uppers}, the number of $k$-multiplexes in $M(G^{\left[d\right]})$ is not greater than $\left( \frac{(kn)!}{k!^n} \right)^{d}$. Consequently, the number of divisible $k$-multiplexes that are the union of a $k_1$-multiplex and a $k_2$-multiplex  is not greater than $\left( \frac{(k_1n)! (k_2n)!}{k_1!^n k_2!^n} \right)^{d}$. Then the number of all divisible $k$-multiplexes is less than
$$\sum\limits_{k_1+k_2 = k} \left( \frac{(k_1n)! (k_2n)!}{k_1!^n k_2!^n} \right)^{d} < k \max\limits_{k_1+k_2 = k} \left( \frac{(k_1n)! (k_2n)!}{k_1!^n k_2!^n} \right)^{d}.$$ 

As is known, if $n \geq 2$ then
$${kn \choose k_1n} >  {k \choose k_1}^n.$$
It holds because if $k = k_1+ k_2$ then the left-hand side of this inequality is a number $kn$-tuples over the set of two symbols in which symbol $i$ appears exactly $k_in$ times and the right-hand side is a number of sets of $n$ $k$-tuples over the same set of symbols in which symbol $i$ appears exactly $k_i$ times.

Using this inequality and comparing the bound on divisible $k$-multiplexes with the number of all $k$-multiplexes we conclude that there exists a constant $0 < \kappa <1$ such that their ratio is not greater than $\frac{k}{c} \kappa^d$ that tends to zero as $d \rightarrow \infty$.

3.  By Lemma~\ref{uppers}, the number of disconnected $k$-multiplexes that are the union of $k$-multiplexes of orders $n_1$ and $n_2$  is not greater than $\left({n \choose n_1} \frac{(kn_1)!(kn_2)!}{k!^n}\right)^d$. Then the number of all disconnected $k$-multiplexes is less than
$$\sum\limits_{n_1+n_2 = k} \left({n \choose n_1} \frac{(kn_1)!(kn_2)!}{k!^n}\right)^d < n \max\limits_{k_1+k_2 = k} \left({n \choose n_1} \frac{(kn_1)!(kn_2)!}{k!^n}\right)^d.$$ 

As is known, if $k \geq 2$ then
$${kn \choose kn_1} > {n \choose n_1}.$$
Using this inequality and comparing the bound on disconnected $k$-multiplexes with the number of all $k$-multiplexes we conclude that there exists a constant $0 < \kappa <1$ such that their ratio is not greater than $\frac{n}{c} \kappa^d$ that tends to zero as $d \rightarrow \infty$.
\end{proof}

Next we obtain the following result for plexes and multiplexes in iterated groups.

\begin{teorema} \label{grouplow}
Let $G$ be a group of order $n$ and let $M(G^{\left[d-1\right]})$ be the $d$-dimensional MDS code of the iterated group $G$.
\begin{enumerate}
\item If $k$ is even and $d \geq 3$ then $M(G^{\left[d-1\right]})$ has a $k$-multiplex. Moreover, there exists a constant $c(G,k) > 0$  such that  
$$\lim\limits_{d \rightarrow \infty} \frac{\Per_k  M(G^{\left[d-1\right]})}{\left( \frac{(kn)!}{k!^n} \right)^{d-2}}  = \lim\limits_{d \rightarrow \infty} \frac{\per_k  M(G^{\left[d-1\right]})}{\left( \frac{(kn)!}{k!^n} \right)^{d-2}} = c(G,k) ,$$
where the latter limit is taken over all $d$ for which $\per_k  M(G^{\left[d-1\right]})$ is nonzero.

\item If $k$ is odd then there exists a constant $c(G,k) > 0$  such that  
$$\lim\limits_{d \rightarrow \infty} \frac{\Per_k  M(G^{\left[d-1\right]})}{\left( \frac{(kn)!}{k!^n} \right)^{d-2}}  = \lim\limits_{d \rightarrow \infty} \frac{\per_k  M(G^{\left[d-1\right]})}{\left( \frac{(kn)!}{k!^n} \right)^{d-2}} = c(G,k) ,$$
where the limits are taken over all $d$ for which $k$-permanents and $k$-multipermanents are nonzero.
\end{enumerate}
\end{teorema}

\begin{proof}
1. By Theorem~\ref{group2plex}, we have that $\per_2  M(G^{\left[2\right]}) >0$. Using Lemma~\ref{plexsvoi}, we obtain that $\Per_k  M(G^{\left[2\right]}) >0$ for all even $k$, and consequently $\Per_k  M(G^{\left[d-1\right]}) >0$ for all even $k$ and odd $d$. As a corollary of Theorem~\ref{quasilowmulti}, we have that for all $d \geq 3$ the $k$-multipermanent of $M(G^{\left[d-1\right]})$ is nonzero and there exists a constant $c(G,k) > 0$  such that  
$$\lim\limits_{d \rightarrow \infty} \frac{\Per_k  M(G^{\left[d-1\right]})}{\left( \frac{(kn)!}{k!^n} \right)^{d-2}}  = c(G,k) ,$$
Since by Theorem~\ref{typicalmulti}, for large $d$ a typical $k$-multiplexes in the iterated group $M(G^{\left[d-1\right]})$ is a $k$-plex, we have that the $k$-permanent of $M(G^{\left[d-1\right]})$ has the same limit.

2. This clause is an immediate corollary of Theorem~\ref{quasilowmulti} and Theorem~\ref{typicalmulti}.
\end{proof}

In conclusion, we note that if the Rodney's conjecture is true then the clause 1 of Theorem~\ref{grouplow} holds for all quasigroups $G$, and if the Ryser's conjecture is true then in addition it holds for all odd $k$ and quasigroups of odd orders $n$.

\section{Proof of the Theorem~\ref{quasilowmulti}}

The main idea of the proof of Theorem~\ref{quasilowmulti} is to show that the number of $k$-multiplexes in the MDS code of an iterated quasigroup is a result of some Markov process.

Suppose $K = \left\{ \alpha^1, \ldots, \alpha^{kn} \right\}$ be a $k$-multiplex in the MDS code $M(G^{\left[d-1\right]})$ of order $n$. Consider the rectangular $kn \times d$  table $T = (t_{i,j})$ such that $t_{i,j} = \alpha^i_j$. Note that the table $T$ has the following two properties:
\begin{enumerate}
\item Each column $T_j$ of $T$ is a permutation of the set $I_{\left\{n,k\right\}}$.
\item For columns $T_j$ of $T$ we have
$$(\ldots((T_1 * T_2)*T_3)* \ldots * T_{d-1} )*T_d = E,$$
where $*$ is the operation of the quasigroup $G$.
\end{enumerate}

Conversely, rows of every rectangular $kn \times d$  table $T$ satisfying these two properties compose a $k$-multiplex in the $(d-1)$-iterated quasigroup $G$.  

Permutations of rows of the table $T$ do not change the corresponding $k$-multiplex $K$, and if $K$ is $k$-plex then it corresponds  to exactly $\frac{(kn)!}{k!^n}$ tables $T$. Since, by Lemma~\ref{uppers}, the number of true $k$-multiplexes is not greater then $\left( n \frac{(kn-2)!}{k!^{n-1} (k-2)!} \right)^{d}$, to prove Theorem~\ref{quasilowmulti} it is suffice to show that the number of such tables is asymptotically equal to $c(G,k) \frac{(kn)!}{k!^n}$ for some constant $c(G,k)$.

Given $m \geq 1$ and a $kn$-vector $U = (u_1, \ldots, u_{kn})$, consider the set $\mathcal{T}_U(m)$ of all $kn \times m$ tables $T$ in which each column $T_j$ is a permutation of the set  $I_{\left\{n,k\right\}}$ and such that
$$(\ldots((T_1 * T_2)*T_3)* \ldots * T_{d-1} )*T_m = U.$$

Let $W$ be a column $kn$-vector that is a permutation of the set $I_{\left\{n,k\right\}}$.  The number of such vectors is equal to the multinomial coefficient ${kn \choose k \ldots k}=\frac{(kn)!}{k!^n}$. For two column $kn$-vectors $U$ and $V$  over the set $I_n^1$ define $a_{U,V}$ to be the number of permutations $W$ such that $V * W = U$. By definition, we have that for given $U$ and $V$ it holds 
$$\sum\limits_{V' \in I_n^{kn}} a_{U,V'} = \sum\limits_{U' \in I_n^{kn}} a_{U',V}  = \frac{(kn)!}{k!^n}.$$  

Note that if we extend  a table $T$  from the set $\mathcal{T}_V(m-1)$ to the right by a permutation $W$ then we obtain a table from the set $\mathcal{T}_U(m)$, where $V * W = U$. 
For every $m$ the number $l_U(m)$ of tables from $\mathcal{T}_U(m)$ depends only on numbers of tables of the previous size and can be expressed as 
$$l_U(m) = \sum\limits_{V \in I_n^{kn}} a_{U,V} l_V (m-1),$$
where   coefficients $a_{U,V}$ are independent of $m$.

If we suppose $L(m)$ to be a $n^{kn}$-vector of all $l_U(m)$, then we obtain
$$L(m) = A(G) L(m-1),$$
where $A(G) = (a_{U,V})$ is a  (0,1)-matrix of order $n^{kn}$. As it was noted before, the sum of entries of the matrix $A(G)$ over each column and each row is equal to $\frac{(kn)!}{k!^n}$.

Next we normalize vectors $L(m)$ and the matrix $A(G)$ and put  $X(m)  = \left(\frac{(kn)!}{k!^n}\right)^{-m} L(m)$ and $B(G) = \left( \frac{(kn)!}{k!^n} \right)^{-1} A(G)$. Then we obtain the process
$$X(m) = B(G) X(m-1)$$
that is the Markov chain with a  doubly stochastic transition matrix $B(G)$ (a nonnegative matrix with sums of entries over each row and each column is equal to 1).  For our purposes we are interested in the value of only one component of the vector $X(m)$, namely $x_E(m)$.

The present reasoning implies that Theorem~\ref{quasilowmulti} is a corollary of the following statement.

\begin{utv} \label{markov}
 Let $G$ be a quasigroup of order $n$. Suppose the matrix $B(G)$ and vectors $X(m)$ are as above and
$$X(m) = B(G) X(m-1).$$
\begin{enumerate}
\item For all even $m$ we have $x_E(m)>0$. If for some odd $m'$ we have $x_E(m')>0$ then $x_E(m)>0$ for all $m \geq m'$.
\item  There exists a constant $c = c(G) >0$ such that  
$$\lim\limits_{m \rightarrow \infty} x_E(m) = c,$$
where limit is taken over all $m$ for which $x_E(m)$ is nonzero.
\end{enumerate}
\end{utv}

To prove this statement we need to recall some concepts and theorems of the theory of Markov processes.  

A Markov chain with a transition matrix $B$ (or the matrix itself) is called \textit{reducible} if by 	simultaneous permutations of rows and columns the matrix $B$ can be conjugate into a block upper triangular matrix of the form 
$$\left( \begin{array} {cc}
 P & R  \\ 0 & S 
\end{array} \right),$$  
where $P$ and $S$ are square matrices. A Markov chain is \textit{irreducible} if it is not reducible.

Let a matrix $B$ be irreducible. We define the \textit{period of index $i$} to be the greatest common divisor of all natural numbers $m$ such that $(i,i)$-th entry of $B^m$ is greater than zero. It is known that for an irreducible matrix $B$ each $i$ has the same period that is called the \textit{period} of $B$. If the period of a matrix $B$ is equal to one then the matrix $B$ is said to be \textit{aperiodic}.  The period of a Markov chain coincides with the period of its transition matrix.

The following theorem is the key result of the theory of Markov chains.

\begin{teorema}[Ergodic theorem] \label{tergo}
A Markov chain
$$X(m) = B X(m-1); ~ x_i(0) \geq 0;~ \sum\limits_{i} x_i(0) = 1$$
 with a transition matrix $B$ is irreducible and aperiodic if and only if for each $i$ there exists $\lim\limits_{m \rightarrow \infty} x_{i}(m) = c_i >0$ that does not depend on the initial state $X(0)$.
\end{teorema}

Now we are ready to prove Proposition~\ref{markov}.

\begin{proof}
Let $G$ be a quasigroup of order $n$, $U,V \in I_n^{kn}$, and let $W$ be a permutation of the multiset $I_{\left\{n,k\right\}}$. Suppose $B(G) = b_{U,V}$ is a nonnegative matrix of order $n^{kn}$ such that $b_{U,V}$ is equal to $\frac{k!^n}{(kn)!}$ multiplied by the number of permutations $W$ for which $V * W = U$. Let us consider the Markov chain
$X(m) = B(G) X(m-1),$ $x_E(0) = 1$, and $x_U(0) = 0$   for all other vectors $U$.  

 We already know that the matrix $B(G)$ is doubly stochastic. 
It is not hard to prove that if a doubly stochastic matrix is reducible then by simultaneous permutations of rows and columns it can be conjugated into a block-diagonal matrix in which each block is a doubly stochastic matrix of smaller order (a proof can be found, for example, in~\cite[p.~34]{minc}).  If the matrix $B(G)$ is reducible then instead of the Markov chain with the matrix $B(G)$ we consider a process with the irreducible block containing entry $b_{E,E}$. So we may assume that $B(G)$ is an irreducible matrix. 

The proof of Lemma~\ref{plexsvoi} implies that the state $x_E(m)$ has period $2$, so matrix $B(G)$ has period at most 2.
Moreover, $B(G)$ is aperiodic if for some odd $m$ we have $x_E(m) >0.$ 

1. If the matrix $B(G)$ is aperiodic, then by the ergodic theorem, there exists a constant $c = c(G)$ such that $\lim\limits_{m \rightarrow \infty} x_E(m) = c.$

2. If the matrix $B(G)$ has period $2$, then we consider the process $X(2t) = B^2(G) X(2t-2),$ $x_E(0) = 1$, and $x_U(0) = 0$   for all other vectors $U$.  The matrix $B^2(G)$ is aperiodic and we have $x_E(m) > 0$ for all even $m$. So, by  the ergodic theorem, there exists a constant $c = c(G)$ such that $\lim\limits_{m \rightarrow \infty} x_E(m) = c$, where limit is taken over even $m$.
\end{proof}

Since we have Proposition~\ref{markov} to be proved, Theorem~\ref{quasilowmulti} and all corollaries  are also proved.

\section{Computational results and concluding remarks}

At the end of this paper, we provide computational results on the numbers of transversals in the Cayley tables of iterated groups and quasigroups of small order, consider a generalization of Theorem~\ref{quasilowmulti} for partial transversals and plexes, and raise several problems.

It is not hard to see that latin hypercubes of orders 2 and 3 are unique up to equivalence and that they are the Cayley tables of the iterated groups $\mathbb{Z}_2$ and $\mathbb{Z}_3$. For order 4 there are plenty of non-equivalent multidimensional latin hypercubes but there exist only two groups of order 4, namely $\mathbb{Z}_4$ and $\mathbb{Z}_2^2$.  The following table summarize results from~\cite{myobz} and~\cite{myquasi} on transversals in $d$-dimensional latin hypercubes that are the Cayley tables of $d$-iterated groups of orders $n \leq 4$.
$$
\begin{array}{|c|c|c|}
\hline
 G \setminus d & \mbox{even} & \mbox{odd} \\
\hline
\mathbb{Z}_2 & 0 & 2^{d-1} \\
\hline
\mathbb{Z}_3 & \frac{2}{3} \cdot 6^{d-1} -  3^{d-2} & \frac{2}{3} \cdot 6^{d-1} +  3^{d-2} \\
\hline
\mathbb{Z}_4 & 0 & \frac{3}{8} \cdot 24^{d-1} + 5 \cdot 8^{d-2} \\
\hline
\mathbb{Z}_2^2 & \frac{3}{8} \cdot 24^{d-1} - 8^{d-2} & \frac{3}{8} \cdot 24^{d-1} + 5 \cdot 8^{d-2} \\
\hline
\end{array}
$$

Let us apply technique proposed in the proof of Theorem~\ref{quasilowmulti} to count the number of transversals in the Cayley tables of iterated group $\mathbb{Z}_5$ and of another iterated quasigroup of order 5.

\subsection{Transversals in the iterated group $\mathbb{Z}_5$}

We start with the group $\mathbb{Z}_5$ whose Caley table is as follows:
$$
\begin{array}{c|ccccc}
 ~ & 1 & 2 & 3 & 4 & 5\\
\hline
 1 & 1 & 2 & 3 & 4 & 5\\
 2 & 2 & 3 & 4 & 5 & 1\\
 3 & 3 & 4 & 5 & 1 & 2\\
 4 & 4 & 5 & 1 & 2 & 3\\
 5 & 5 & 1 & 2 & 3 & 4\\
\end{array} 
$$

As before, let $l_U(d)$ be the number of tables $T$ from the set $\mathcal{T}_U(d)$ such that each their column $T_j$ is a permutation of the set of elements of $\mathbb{Z}_5$ and 
$$(\ldots((T_1 * T_2)*T_3)* \ldots )*T_d = U,$$
where $*$ is the $\mathbb{Z}_5$ group operation.

Note that if there exists a permutation $\sigma \in S_5$ such that $(u_1, \ldots, u_5) = (v_{\sigma(1)}, \ldots, v_{\sigma(5)})$ then the numbers $l_U(d)$ and $l_V(d)$ are the same. Also, if vectors $U$ and $V$ satisfy $U * H = V$ where $H = (h, \ldots, h)$, $h \in \mathbb{Z}_5$ then $l_U(d) = l_V(d)$. This fact is easy to prove by induction on $d$ using associativity of the group $\mathbb{Z}_5$ and the fact that for any permutation $W$ of the group $\mathbb{Z}_5$ and any $H = (h, \ldots, h) $ the vector $H * W$ is a permutation. We will say that vectors $U$ and $V$ are equivalent if they can be turn to each other by these two operations. 

Denote by $Z_U$ the equivalence class for vector $U$. By direct calculation, it can be checked that $l_U(d)$ is non zero for some $d$ only if vector $U$ belongs to one of the following equivalence classes:
$$Z_{11111}, ~ Z_{12345},~Z_{11134},~Z_{11125},~Z_{11224},~Z_{11332}.$$ 

Also, it can be verified that numbers $l_U(d)$ for $U$ from $Z_{11134}$ and $Z_{11125}$ satisfy the similar relations, so they are coincide and these two classes can be joined into one. The same is true for classes $Z_{11224}$ and $Z_{11332}$. For shortness, we introduce the following notations
$$y_1(d) = \sum\limits_{U \in Z_{11111}} l_U(d); ~~~ y_2(d) = \sum\limits_{U \in Z_{12345}} l_U(d);$$
$$y_3(d) = \sum\limits_{U \in Z_{11134}, ~ Z_{11125}} l_U(d); ~~~ y_4(d) = \sum\limits_{U \in Z_{11224}, ~ Z_{11332}} l_U(d).$$

Then the number of transversals in the $d$-dimensional latin hypercube that is the Cayley table of $d$-iterated group $\mathbb{Z}_5$ is equal to $\frac{1}{5 \cdot 120}y_1(d+1)$ or to $\frac{1}{120} y_2(d)$.  
By direct calculations, we establish that $y_i(d)$ satisfy the following recurrence relations:
$$
\left( 
\begin{array}{c}
 y_1(d) \\
y_2(d) \\
y_3(d) \\
y_4(d) \\
\end{array} 
\right)
=
\left(
\begin{array}{cccc}
0 & 5 & 0 & 0 \\
120 & 15 & 30 & 20 \\
0 & 50 & 30 & 40 \\
0 & 50 & 60 & 60 \\
\end{array} 
\right)
\left(
\begin{array}{c}
 y_1(d-1) \\
y_2(d-1) \\
y_3(d-1) \\
y_4(d-1) \\
\end{array} 
\right);
$$
$$ y_1(0) = 1;~y_2(0) = y_3(0) = y_4(0) = 0. $$

Note that these relations define a process with irreducible and aperiodic matrix $A$.  So we have that for all $d \geq 1$ the number of transversals in the $d$-iterated group $\mathbb{Z}_5$ is greater than zero and there exists a constant $c(\mathbb{Z}_5, 1)$ for which
$$\lim\limits_{d \rightarrow \infty} \frac{T  (Q(G^{\left[d\right]}))}{ 120^{d-1}} = c(\mathbb{Z}_5, 1).$$

To find the limit constant $c(\mathbb{Z}_5, 1)$ we use the fact that $y_1(d) =\chi_1 120^d$ where $\chi_1$ is the first component of the eigenvector $\chi$ of the matrix $A$ corresponding to the largest eigenvalue $\lambda = 120$ and normalized so that $\sum\limits_{i=1}^4 \chi_i = 1$.   Compute that
$$\chi = \frac{1}{125} (1, 24, 40, 60)^T.$$
Thus $\chi_1 = \frac{1}{125}$ and $c(\mathbb{Z}_5,1) = \frac{120\chi_1}{5} = \frac{24}{125}.$ Therefore, for the number of transversals in the $d$-dimensional $d$-iterated group $\mathbb{Z}_5$ and for the permanent of the corresponding $(d+1)$-dimensional MDS code we have
 $$\lim\limits_{d \rightarrow \infty} \frac{T  (Q(G^{\left[d\right]}))}{ 120^{d-1}} = \frac{24}{125}$$
and
 $$\lim\limits_{d \rightarrow \infty} \frac{\per  M(G^{\left[d\right]})}{ 120^{d-1}} = \frac{24}{125}.$$

\subsection{Transversals in the quasigroup of order 5}

Let us consider the quasigroup $G$ of order 5 defined by the following Cayley table:
$$
\begin{array}{c|ccccc}
 ~ & 1 & 2 & 3 & 4 & 5\\
\hline
 1 & 1 & 2 & 3 & 4 & 5\\
 2 & 2 & 1 & 4 & 5 & 3\\
 3 & 3 & 4 & 5 & 1 & 2\\
 4 & 4 & 5 & 2 & 3 & 1\\
 5 & 5 & 3 & 1 & 2 & 4\\
\end{array} 
$$

We will use the same notations as before but now $Z_U$ will be the set of vectors $V$ that can be turn to the vector $U$ only via a permutation of components. 
 By direct calculation, it can be checked that for all vectors $U$ the number $l_U(d)$ is not zero for certain $d$. 

Analyzing  the relations between  numbers $l_U(d)$, we divide all vectors from the set $I_5^5$ into the following classes (different letters mean different values):
$$H_1 = \cup Z_{iiiii}; ~~~ H_2 = Z_{12345}; ~~~ H_3 = \bigcup Z_{222ij} \cup Z_{iii2j};$$
$$ H_4 = \cup Z_{iiijk}, \mbox{ where } i,j,k \neq 2; ~~~ H_5 = \bigcup Z_{222ii}  \cup Z_{iii22}; ~~~ H_6 = \cup Z_{iiijj}, \mbox{ where } i,j \neq 2;$$
$$H_7 = \cup Z_{ii2jk}; ~~~ H_8 = \cup Z_{iijkl}, \mbox{ where } i,j,k,l \neq 2; ~~~ H_9 = \cup Z_{22ijk};$$
$$H_{10} = \cup Z_{22iik}; ~~~ H_{11} = \cup Z_{iijjk}, \mbox{ where } i,j,k \neq 2;; ~~~ H_{12} = \cup Z_{iijj2}; ~~~ H_{13} = \cup Z_{iiiij}.$$

Put $y_i(d) =  \sum\limits_{U \in H_{i}} l_U(d)$ for all $i = 1, \ldots, 13$ and denote $Y(d) = (y_1(d), \ldots, y_{13}(d))^T.$ By calculations, we found that $Y(d)$ satisfies the process $Y(d) = A \cdot Y(d-1)$ with the matrix
$$ A = \left(
\begin{array}{ccccccccccccc}
0 & 5 & 0 & 0 & 0 & 0 & 0 & 0 & 0 & 0 & 0 & 0 & 0 \\
120 & 3 & 6 & 12 & 12 & 0 & 2 & 2 & 8 & 0 & 4 & 12 & 0 \\
0 & 18 & 18 & 18 & 0 & 0 & 12 & 18 & 18 & 12 & 12 & 24 & 0 \\
0 & 24 & 12 & 12 & 0 & 0 & 8 & 6 & 6 & 12 & 12 & 8 & 0 \\
0 & 8 & 0 & 0 & 0 & 0 & 4 & 0 & 0 & 8 & 8 & 0 & 0 \\
0 & 0 & 0 & 0 & 0 & 0 & 6 & 6 & 6 & 8 & 8 & 8 & 0 \\
0 & 12 & 24 & 24 & 36 & 36 & 30 & 24 & 24 & 28 & 28 & 16 & 72 \\
0 & 4 & 12 & 6 & 0 & 12 & 8 & 12 & 6 & 12 & 8 & 8 & 24 \\
0 & 16 & 12 & 6 & 0 & 12 & 8 & 6 & 0 & 16 & 12 & 0 & 24 \\
0 & 0 & 12 & 18 & 36 & 24 & 14 & 18 & 24 & 4 & 8 & 24 & 0 \\
0 & 12 & 12 & 18 & 36 & 24 & 14 & 12 & 18 & 8 & 12 & 16 & 0 \\
0 & 18 & 12 & 6 & 0 & 12 & 4 & 6 & 0 & 12 & 8 & 4 & 0 \\
0 & 0 & 0 & 0 & 0 & 0 & 10 & 10 & 10 & 0 & 0 & 0 & 0 \\
\end{array} 
\right).$$

The eigenvector $\chi$ of the matrix $A$ corresponding to the largest eigenvalue $\lambda$ with $\sum\limits_{i=1}^{13} \chi_i = 1$ is 
$$\chi = \frac{1}{625} (1, 24, 72, 48, 16, 24, 144, 48, 48, 72, 72, 36, 20)^T.$$

So the limit constant $c(G,1)$ for the number of transversals in the $d$-dimensional $d$-iterated quasigroup $G$ and for the permanent of the corresponding $(d+1)$-dimensional MDS code is $c(G,1) = \frac{24}{625}.$

\subsection{Remarks and open questions}

First of all we would like to note that the technique proposed in this paper can be used not only to count transversals in considered iterated quasigroups of order 5 but  to estimate the number of transversals and $k$-plexes in other small-ordered iterated quasigroups and to find these numbers for small dimensions.

Secondly, we note that $d$-iterated quasigroups obtained from isotopic binary quasigroups can be non-isotopic and can have different number of transversals and plexes. $d$-ary quasigroups $f$ and $g$ of order $n$ are called \textit{isotopic}, if there exist permutations $\sigma_i \in S_n$, $i = 0, \ldots, d$ such that 
$$f(x_1, \ldots, x_d) \equiv \sigma^{-1}_0 \left( g(\sigma_1(x_1), \ldots, \sigma_d(x_d)) \right).$$ 
For example, there exists a binary quasigroup $G$ of order 4 isotopic to the group $\mathbb{Z}_4$ but in contrast to the $d$-iterated group $\mathbb{Z}_4$ the $d$-iterated quasigroup $G$ has transversals for all  $d \geq 3$.

At last, acting like in the proof of Theorem~\ref{quasilowmulti}, we can obtain the similar theorem for partial structures in iterated quasigroups.
A \textit{partial diagonal of length $t$} in a $d$-dimensional matrix of order $n$ is a set of $t$ indices such that each hyperplane contains no more than one of chosen indices.    A \textit{partial transversal of length $l$} in a $d$-dimensional latin hypercube $Q$ of order $n$ is a set of indices that corresponds to a unity partial diagonal in the $(d+1)$-dimensional matrix $M(Q)$.

A \textit{partial $k$-multiplex of length $l$} in a multidimensional matrix is a multiset of $kl$ indices such that there are either $k$ or zero indices in each hyperplane. Analogically, we can define partial $k$-multiplexes in latin hypercubes.

For a $d$-dimensional nonnegative matrix $A$ of order $n$ let us denote by $P_{l,k}(A)$ the number of partial $k$-multiplexes of length $l$ over the support of $M$. Using the same method as in the proof of Theorem~\ref{quasilowmulti}, we can prove the following statement.

\begin{utv}
Let $G$ be a binary quasigroup of order $n$ and let $M(G^{\left[d-1\right]})$ be the $d$-dimensional MDS code of the  $(d-1)$-iterated quasigroup $G$.
\begin{enumerate}
\item For all even $d$ and for all $l \leq n$ the MDS code $M(G^{\left[d-1\right]})$ has a partial $k$-multiplex of length $l$. If for some odd $d'$ we have that $P_{l,k} (M(G^{\left[d'-1\right]}))$ is positive then $P_{l,k} (M(G^{\left[d-1\right]})) $ is greater than zero for all $d \geq d'$.
\item  There exists a constant $c(G, k, l) >0$ such that  
$$\lim\limits_{d \rightarrow \infty} \frac{P_{l,k}  (M(G^{\left[d-1\right]}))}{\left( {n \choose l} \frac{(kl)!}{k!^l} \right)^{d-2}} = c (G, k, l) ,$$
where limit is taken over all $d$ for which $M(G^{\left[d-1\right]})$ contains a partial $k$-multiplex of length $l$.
\end{enumerate}
\end{utv}      

In conclusion, we raise several questions and open problems inspired by the obtained results:

1. How can the limit constants $c(G,k)$ for the number of $k$-plexes be estimated in general and for given groups and quasigroups $G$? 

2. How are the limit constants $c(G,k)$ for the number of $k$-plexes in a given iterated quasigroup $G$  related to each other for different $k$?   

3. Does the iterated groups $\mathbb{Z}_n$ always have the maximal number of transversals (and possibly $k$-plexes and $k$-multiplexes) among  all iterated quasigroups of order $n$?

4. Let $k \leq n^{d-1}$ and a latin hypercube $Q$ has a $k$-multiplex. Is it true that $Q$ has a $k$-plex? The similar question can be asked for a nonnegative matrix.

\end{document}